\documentclass[11pt]{amsart}

\usepackage{cite}
\usepackage{amssymb}
\usepackage{amsmath}
\usepackage{amsthm}
\usepackage{amsfonts}
\usepackage{graphicx}
\usepackage{bbm}
\usepackage{graphicx}%
\usepackage{amscd}
\usepackage{color}
\usepackage[usenames,dvipsnames,svgnames,table]{xcolor}
\usepackage[toc,page]{appendix}
\usepackage{bbm,cancel}
\usepackage[normalem]{ulem}
\usepackage{mathtools}
\usepackage{comment}
\usepackage{bm,bbm}
\usepackage{arydshln}
\usepackage[shortlabels]{enumitem}
\usepackage{hyperref}
\usepackage{cleveref}

\makeatletter
\@namedef{subjclassname@2020}{%
  \textup{2020} Mathematics Subject Classification}
\makeatother

\newtheorem{theorem}{Theorem}[section]

\newtheorem{lemma}[theorem]{Lemma}

\newtheorem{corollary}[theorem]{Corollary}

\newtheorem{example}{Example}[section]

\newtheorem{definition}{Definition}[section]
\newtheorem*{example*}{Example}

\theoremstyle{remark}
\newtheorem*{remark}{Remark}

\title{Asymptotics of powers of random elements of compact Lie groups}

\begin{document}

\author[Phillips]{Donnelly Phillips{$^{\dag }$}}
\thanks{\footnotemark {$^\dag$} Research was supported in part by NSF Grant DMS-1255574.}
\address{Department of Mathematics \\
University of Virginia\\
Charlottesville, VA 22904, U.S.A.}
\email{dkp7tpj@virginia.edu}

\keywords{random matrix, high powers, limiting distribution, compact Lie group}
\subjclass[2020]{Primary 60B20}

\maketitle

\begin{abstract}
For a Haar-distributed element $H$ of a compact Lie group \(L\), Eric Rains proved in \cite{rains} that there is a natural number $D = D_L$ such that, for all $d\ge D$, the eigenvalue distribution of $H^d$ is fixed, and Rains described this fixed eigenvalue distribution explicitly. In the present paper we consider random elements $U$ of a compact Lie group with general distribution. In particular, we introduce a mild absolute continuity condition under which the eigenvalue distribution of powers of $U$ converges to that of $H^D$. Then, rather than the eigenvalue distribution, we consider the limiting distribution of $U^d$ itself.
\end{abstract}

\tableofcontents

\section{Introduction}

Let \(\mathbb{U}(N)\) be the group of \(N \times N\) unitary matrices, and let \(H\) be a Haar-distributed random element of \(\mathbb{U}(N)\). A topic of high interest in random matrix theory is understanding the patterns in the eigenvalues of \(H\), some of which arise in raising \(H\) to a power.  The joint density of the eigenvalues of \(H\) on \((\mathbb{S}^1)^N\) with respect to Haar measure on \((\mathbb{S}^1)^N\) is given by the Weyl integration formula as
\[\rho(z_1, \ldots , z_N) \:=\: \prod_{1\leq j < k \leq N} |z_j - z_k|^2 \, ;\]
see for example \cite{eliz}. In \cite{rains}, Rains discussed that, because this formula is a Laurent polynomial in \(N\) variables, and is of degree \(N-1\) in any one variable, the eigenvalues of \(H^{N}\) are iid uniform random variables on \(\mathbb{S}^1\).

In fact, Rains showed that this occurs more generally \cite[Theorem 2.1]{rains}.

\begin{theorem}\label{rains1}
Let \(L\) be a compact Lie group that has a maximal torus \(T\) of dimension \(n\), and let \(\phi : L \rightarrow GL(V)\) be a unitary representation of \(L\) on  a vector space \(V\) of dimension \(N\). Let \(C \subseteq L\) be a connected component of \(L\), and suppose that \(H\) is Haar-distributed on \(C\). Let \(\{Y_j\}_{1 \leq j\leq n}\) be iid uniform random variables on \(\mathbb{S}^1\). Then there exist a set of Laurent monomials \(\{p_j\}_{1 \leq j \leq N}\) on \(\mathbb{C}^n\) and \(D = D_L \in \mathbb{N}\) such that, for all \(d \geq D\),  the eigenvalue distribution of \(\phi(U^d)\) has the distribution of the tuple \((p_j(Y_1, \ldots , Y_n))_{1 \leq j \leq N}\).
\end{theorem}

For example, let \(L = C = \mathbb{U}(N)\). Then for \(d \geq D := N\), the eigenvalue distribution of \(H^d\) is the distribution of \((Y_1, \ldots , Y_N)\). When \(L = C = \mathbb{SU}(N)\), then for \(d \geq D := N\), the eigenvalue distribution of \(H^d\) is the distribution of \((Y_1, \ldots, Y_{N-1}, \overline{Y_1 \ldots Y_{N-1}})\). For \(N = 2n+1\), \(L = C = \mathbb{SO}(N)\), for \(d \geq D := N\), the eigenvalue distribution of \(H^d\) is the distribution of \((Y_1, \ldots, Y_n, \overline{Y_1} , \ldots , \overline{Y_n}, 1)\).

From these examples, we see that we need the monomials \(\{p_j\}_{1 \leq j \leq N}\) partly to model the fact that the eigenvalues may have restrictions or be limited in their degrees of freedom. Intuitively, this means that \(H\) raised to a sufficiently high power has eigenvalues that are as uniformly distributed as possible.

There is much more that can be said regarding patterns in the eigenvalues of Haar measure arising from powers. For \(H\) Haar-distributed on a compact connected Lie group, Rains in \cite{rains2} explicitly describes the eigenvalue distributions of \(H^d\) for \(d < D\). In \cite{tracepowers} Diaconis and Shahshahani show that for \(H\) Haar-distributed on \(\mathbb{U}(N)\),  the trace of \(H^m\) is asymptotically normally distributed as \(N \rightarrow \infty\). There is also interest in analyzing powers, traces, and moments for other distributions on \(\mathbb{U}(N)\), such as unitary Brownian \cite{ubm_levy_schurweyl}, \cite{ubm_dalq_weincalc}.

In this paper, we present results that are related to the results stated above, but different in nature. In works like \cite{rains} and \cite{rains2}, the focus lies in using algebraic methods to derive exact formulas for powers of random Lie group elements with specific distributions. In this work, we will instead use analysis techniques to show that, for a wide class of random elements with  a weak continuity-type assumption, an asymptotic phenomenon occurs. Also in contrast to other works, we discuss the limit of the random group element itself.

\emph{Acknowledgement.} The author would like to recognize his advisor Tai Melcher for her many comments and suggestions. He would also like to thank Mark Meckes and Todd Kemp for their helpful remarks in this paper's earliest stage. The author is grateful for receiving support from the NSF.

\subsection{Background and notation}

We set here the background and notation for  compact Lie groups, much of which is based on the presentation in \cite{eliz} and \cite{hall}. A \emph{Lie group} \(L\) is a differentiable manifold that is also a group, such that the multiplication and inverse operators are smooth maps. For simplicity, we will restrict to the case when \(L\) is connected. For a Lie group \(L\), there exists a translation-invariant measure \(\mu_L\), such that for all \(g \in L\) and measurable \(A \subseteq L\), \(\mu_L(gA) = \mu_L(A) = \mu_L(Ag)\). Moreover, this measure is unique up to scalar multiple. When \(L\) is compact, \(\mu_L(L) < \infty\), and thus we may normalize \(\mu_L\) and assume \(\mu_L(L) = 1\). Being absolutely continuous with respect to Haar measure intuitively means that there is no ``weight'' or ``mass'' being assigned to submanifolds of smaller dimension, and in particular means that the distribution does not have point-masses. For example, a random matrix in \(\mathbb{SU}(N)\) would not have distribution that is absolutely continuous with respect to \(\mu_{\mathbb{U}(N)}\), nor would a random matrix that took on countably-many values with probability 1.

A \emph{representation} is a continuous homomorphism \(\phi: L \rightarrow GL(V)\) for some vector space \(V\). If \(V\) is equipped with an inner-product, we say \(\phi\) is unitary if \(\phi(g)\) is unitary for all \(g \in L\). It is a theorem (\cite[Proposition 2.2.1]{applebaum}) that, given a representation of a compact Lie group \(\phi : L \rightarrow GL(V)\), there exists an inner product on \(V\) that makes \(\phi\) a unitary representation. Thus, without loss of generality, we will assume that our representations are unitary. Compact Lie groups can always be realized as closed subgroups of \(\mathbb{U}(N)\) for some \(N\), and as such the results of this paper can be interpreted as statements about random unitary matrices.

Every compact Lie group \(L\) has a \emph{maximal torus}, which is a Lie subgroup that is isomorphic as a Lie group to a torus, and is maximal under set inclusion over all such Lie subgroups. For example, in \(\mathbb{U}(N)\), the subset of all diagonal matrices is a maximal torus. All elements of \(L\) lie in a maximal torus, and all maximal tori are conjugate to each other. It follows that given a maximal torus \(T\), for any element \(g \in L\),  there exists an element in \(T\) conjugate to \(g\).

What follows can be referenced in chapter 11 of \cite{hall}. This paper will include discussion of the space \(L/T\). Since we cannot expect \(T\) to be normal in \(L\), \(L/T\) is not a group, but we may still regard it as a collection of cosets, and it is a compact smooth manifold. 
Furthermore, \(L\) naturally acts on \(L/T\) on the left via \(g \cdot aT = gaT\). \(L/T\) has an analogue of Haar measure, \(\mu_{L/T}\), which is invariant under the action of \(L\), so that for all \(g \in L\) and measurable \(A \subseteq L/T\), \(\mu_{L/T}(gA) = \mu_{L/T}(A)\).
We define the smooth map \(\psi : L/T \times T \rightarrow L\) as 
\begin{equation}\label{psimap}
\psi(v T, t) = vtv^{-1}   \,.
\end{equation}
It is known that, on a set of full measure, \(\psi\) is a \([N(T):T]\)-to-1 map (where \(N(T)\) is the normalizer of \(T\) in \(L\)) with bijective differential \(d\psi\).

To summarize the notation that will be used, we assume that \(L\) is a compact, connected Lie group, \(T\) is a maximal torus, and \(U\) is a random element of \(L\). Let \(\mu_L\) denote Haar measure on \(L\), and let \(H_L\) denote a Haar-distributed element of \(L\). Let \(D \in \mathbb{N}\) denote the exponent described in Theorem \ref{rains1}, so that the eigenvalues of \((H_L)^d\) are distributed as in Theorem \ref{rains1} for any \(d \geq D\).

\subsection{Statement of main results}

To state the main results, we first offer a definition. We say that the random pair \((U_{L/T} , U_T) \in L/T \times T\) is a random preimage of \(U\) if \((U_{L/T}, U_T) \in \psi^{-1}(U)\) almost surely. This definition will be restated in Definition \ref{def}, along with discussion and examples. We also fix a Lie group isomorphism \(z = (z_1, \ldots , z_n) : T \rightarrow (\mathbb{S}^1)^n\).


Section 2 will be devoted to demonstrating the following theorem.

\begin{theorem}\label{mainevalthm}
Let \(\phi\) be a unitary representation of \(L\) on a vector space \(V\) of dimension \(N\) and \(T\) a maximal torus of dimension \(n\). For \(U\) a random element of \(L\), let \((U_{L/T}, U_T)\) be a random preimage of \(U\). If \(Law(U_T) \ll \mu_T\) (for example, if \(Law(U) \ll \mu_L\)), then the eigenvalues of \(\phi(U^m)\) converges in distribution to the eigenvalues of \(\phi(H_L^D)\) as \(m \rightarrow \infty\). Moreover, if the marginal density of \(U_T\) can be expressed as a finite Laurent polynomial in the \(z_j\)'s in which the degree in any \(z_j\) is at most \(M\), then the eigenvalues of \(\phi(U^m)\) are exactly equal in distribution to the eigenvalues of \(\phi(H_L^D)\) for \(m > M\).
\end{theorem}

The first part of this theorem discusses the behavior for general distributions on \(L\) under power maps, and is proven in this work with analysis techniques. On the other hand, the second part of this theorem will follow from an observation already made in \cite{rains}, which is restated in Remark \ref{lem3p15}, and will apply to uniformly-distributed \(L\)-valued random elements. Indeed, the Weyl integral formula tells us that the density of the eigenvalues of \(H_L\) will consist of a finite Laurent polynomial, and Rains showed that this is also the case for a Haar-distributed element of a connected component of a disconnected compact Lie group.

We will also discuss the limiting distribution \(U^m\) itself. In section 3, we will prove the following.
\begin{theorem}\label{lastthm}
If \(Law(U_T) \ll \mu_T\), then \(U^m\) converges in distribution to \(\psi(U_{L/T}, Y)\), where \(Y\) is distributed as \(\mu_T\) and independent of \(U_{L/T}\).
\end{theorem}

As will be pointed out in Remarks \ref{thmnotobv1} and \ref{thmnotobv2},  it is interesting that the weak convergence in Theorem \ref{mainevalthm} and Theorem \ref{lastthm} holds in spite of any dependence among the eigenvalues or with \(U_{L/T}\).

Section 4 includes a few results illustrating that the conditions in Theorems \ref{mainevalthm} and \ref{lastthm} are independent of the choice of random preimage.

As a final remark before starting the proofs, it should be apparent that these results will not hold for arbitrary random elements of \(L\). For example, if the distribution of \(U\) consisted of point masses, then the distribution of powers of \(U\) would also have point masses, as would that of the eigenvalues, and the desired convergence could never be achieved. However, the condition that this paper describes is actually weaker than the absolute continuity of the distribution of \(U\) with respect to Haar measure on \(L\), which will be illustrated in Example \ref{lastexample}.

\section{Limiting distributions of eigenvalues}

For this section, we let \(X_1, \ldots , X_n \in [0,2\pi )\) be random angles, and let \(\nu\) be the distribution of \(X = (X_1, \ldots , X_n)\).  Then for \(m \in \mathbb{N}\), we will denote the distribution of \(mX = (mX_1 , \ldots , mX_n)\) as \(\nu^{(m)}\), where \(mX_j\) is interpreted modulo \(2\pi\). Let \(\mu_{[0,2\pi)^n}\) denote Lebesgue measure on \(\mathbb{R}^n\) restricted to \([0,2\pi)^n\). We will demonstrate the following theorem.

\begin{theorem}
\label{torus1}
If \(\nu \ll \mu_{[0,2\pi)^n}\), then \(\nu^{(m)}\) converges weakly to \(\mu_{[0,2\pi)^n}\).
\end{theorem}

\begin{remark}\label{thmnotobv1}
Phrased another way, Theorem \ref{torus1} states that \((mX_1, \ldots , mX_n)\) converges in distribution to \((Y_1, \ldots , Y_n)\), where \(\{Y_j\}_{j=1}^n\) are iid uniform random variables on \([0,2\pi)\). Raising elements of \(\mathbb{S}^1\) to high powers will effectively spin them around \(\mathbb{S}^1\), and one might see how spinning a single random value would give rise to a uniform distribution on the circle, but it might be less obvious why independence occurs in the limit when several random values are spun at once. The key is to consider how the described spinning affects the joint distribution.
\end{remark}

We will provide two proofs of Theorem \ref{torus1}. In Section 2.1, we use Fourier analysis, while in Section 2.2, we will use elementary measure theory, which provides a more physical understanding how the distribution \(\nu^{(m)}\) changes in \(m\). Section 2.3 will apply this result to prove Theorem \ref{mainevalthm}.

\subsection{Fourier analysis approach}

Given a measure \(\alpha\) on \([0,2\pi)^n\) (or an element \(f \in L^1([0,2\pi)^n)\)), we let \(\widehat{\alpha}\) (resp. \(\widehat{f}\)) denote its Fourier transform. We first prove the following general result.

\begin{lemma}\label{fourierlemma}
For \(\nu\), and \(\nu^{(m)}\) as described above, and for any \((p_1, \ldots , p_n) \in \mathbb{Z}^n\),
\[\widehat{\nu^{(m)}}(p_1, \ldots , p_n) = \widehat{\nu}(m p_1, \ldots , mp_n) \, .\]
\end{lemma}

\begin{proof}
\begin{align*}
\widehat{\nu^{(m)}}(p_1, \ldots , p_n) &\:=\: \int_{[0,2\pi)^n} \left(\prod_{1\leq j\leq n} e^{-ip_j \theta_j} \right) \nu^{(m)}(d \theta) \\
& \:=\: \mathbb{E} \left[ \prod_{1 \leq j \leq n} e^{-i p_j m X_j } \right] \\
& \:=\: \int_{[0,2\pi)^n} \left(\prod_{1 \leq j \leq n} e^{-i p_j m \theta_j} \right) \nu(d\theta) \:=\: \widehat{\nu}(m p_1, \ldots , m p_n) \, .
\end{align*}
\end{proof}

\begin{remark}\label{lem3p15}
This lemma rationalizes the stationary behavior exhibited in Rains' Theorem \ref{rains1} (and is explained well in \cite[Lemma 3.15]{eliz}). Indeed, suppose that \(\rho\) is the density of \(\nu\) with respect to \(\mu_{[0,2\pi)^n}\), and \(\rho^{(m)}\) is the density of \(\nu^{(m)}\). Then if \(\rho\) is the finite Fourier series
\[\rho(\vec{\theta}) \:=\: \sum_{\vec{p} \in \mathbb{Z}^n} a_{\vec{p}} \prod_{1 \leq j \leq n} e^{i p_j \theta_j}   \,,\]
then   \(\rho^{(m)}\) is also a finite Fourier series, but with fewer nonzero terms, namely
\[\rho^{(m)}(\vec{\theta}) \:=\: \sum_{\substack{\vec{p} \in \mathbb{Z}^n \\ m | p_j \: \forall j}} a_{\vec{p}}  \prod_{1 \leq j \leq n} e^{i \frac{p_j}{m} \theta_j}   \, .\]
Thus, if \(m\) is greater than \(\max\{p_j \::\: 1 \leq j \leq n \:,\: a_{\vec{p}} \neq 0\}\), then we see that the Fourier series of \(\rho^{(m)}\) will consist of merely the constant term 1.

This observation on its own would not be suitable to prove Theorem \ref{torus1}, as it does not account for infinite Fourier series or the type of convergence of such series. Instead, we use the following argument.
\end{remark}

\begin{proof}[Proof of Theorem \ref{torus1}.]
Let \(\rho\) be the \(L^1\) density of \(\nu\) with respect to \(\mu_{[0,2\pi)^n}\). By the Riemann-Lebesgue lemma, \(\widehat{\rho} \in c_0(\mathbb{Z}^n)\). In other words, for any sequence \((\vec{p}^{(\ell)})_\ell \subseteq \mathbb{Z}^n\) with \(\|\vec{p}^{(\ell)}\| \rightarrow \infty\), \(\widehat{\nu}(\vec{p}^{(\ell)}) = \widehat{\rho}(\vec{p}^{(\ell)}) \rightarrow 0\) as \(\ell \rightarrow \infty\). For fixed \(\vec{p} \in \mathbb{Z}^n \setminus \{\vec{0}\}\) by Lemma \ref{fourierlemma}, \(\widehat{\nu^{(m)}}(\vec{p}) = \widehat{\nu}(m \vec{p})\), and as \(m \rightarrow \infty\), \(\widehat{\nu}(m \vec{p}) \rightarrow 0\), while \(\widehat{\nu^{(m)}}(\vec{0}) = 1\) for all \(m\), so that ultimately \(\widehat{\nu^{(m)}}(\vec{p}) \rightarrow \delta_{\vec{0}}(\vec{p}) = \widehat{\mu_{[0,2\pi)^n}}(\vec{p})\). Since  convergence of the Fourier coefficients provides convergence of the measures (\cite[Theorem 4.2.5]{applebaum}, for example), \(\nu \rightarrow \mu_{[0,2\pi)^n}\) weakly.
\end{proof}

\subsection{Standard measure theory approach}

We now begin the second proof of Theorem \ref{torus1}, which relies on examining more closely how the distributions \(\nu^{(m)}\)  change as \(m\) increases. The transformation below describes how the density of \(\nu^{(m)}\) changes in \(m\) specifically for \(n = 1\).

\begin{lemma}\label{rlem}
For any \(m \in \mathbb{N}\), the transformation \(R^{(m)} : L^1([0,2\pi)) \rightarrow L^1([0,2\pi))\) defined as 
\[(R^{(m)} f) (x) \:=\: \frac{1}{m} \sum_{k =0}^{m-1} f \left( \frac{x + 2\pi k}{m} \right)\]
preserves the value of the integral, so that \(\int_{[0,2\pi)} R^{(m)} f = \int_{[0,2\pi)} f\), and \(R^{(m)}\) is a contraction. Furthermore, if \(\rho \in L^1([0,2\pi))\) is a probability density for \(X\), then \(R^{(m)} \rho\) is a probability density for \(mX\).
\end{lemma}

\begin{proof}
Using a change of variables, we may directly calculate that
\begin{align*}
\int_{[0,2\pi)} \frac{1}{m} \sum_{k=0}^{m-1} f\left(\frac{x+2\pi k}{m}\right)dx & \:=\: \sum_{k=0}^{m-1} \int_{[0,2\pi)} \frac{1}{m} f \left( \frac{x + 2\pi k}{m} \right)dx \\
& \:=\: \sum_{k=0}^{m-1} \int_{\left[2\pi \frac{k}{m}, 2\pi \frac{k+1}{m}\right)} f(u) du \\
& \:=\: \int_{[0,2\pi)} f(u) du \, .
\end{align*}

To show that \(R^{(m)}\) is a contraction map, observe that by the triangle inequality, \(|R^{(m)} f| \leq R^{(m)} |f|\), so that \(\int |R^{(m)} f| \leq \int R^{(m)} |f| = \int |f|\).

Lastly, for a measurable subset \(A \subseteq [0,2\pi)\), let \(m^{-1}(A)\) denote the inverse image of \(A\) under the map \(x \mapsto mx \: (\text{mod } 2\pi)\). Then it can be seen that
\begin{align*}
m^{-1}(A) &\:=\: \{x \in [0,2\pi) \::\: mx \in A\} \\
& \:=\: \frac{1}{m}A \cup \bigg(\frac{1}{m}A + \frac{2\pi}{m} \bigg) \cup \ldots \cup \bigg(\frac{1}{m}A + \frac{2\pi(m-1)}{m} \bigg) \, ,
\end{align*}
which is necessarily a disjoint union (finding these disjoint sets directly corresponds to finding \(m\)th roots on \(\mathbb{S}^1\)). Then
\[P(mX \in A) \:=\: P(X \in m^{-1}(A)) \:=\: \int_{m^{-1}(A)} \rho \:=\: \int_A R^{(m)} \rho \, ,\]
where the last equality follows again by a  change of variables. Thus, \(R^{(m)}\rho\) is a density of \(mX\).

\end{proof}

To offer some intuition for this proof of Theorem \ref{torus1}, for \(f \in L^1([0,2\pi))\), and for any \(x \in [0,2\pi)\), \(R^{(m)} \rho(\vec{x})\) can be thought of as an approximation of the Riemann integral of \(f\), even though \(f\) need not be Riemann integrable. This issue can be circumvented by approximating \(f\) by a  Riemann integrable function.

Recall that a simple function is a sum of characteristic functions \(\sum_{j=1}^k \chi_{E_j}\). A step function \(s : \mathbb{R}^n \rightarrow \mathbb{R}\) is a simple function in which the sets are rectangles, that is, \(E_j = [a_1, b_1] \times \ldots \times [a_n, b_n] \subseteq \mathbb{R}^n\) for \(a_j, b_j \in \mathbb{R}\). The following is a well-known lemma, a version of which can be found in \cite[Theorem 4.3]{stein} (while the cited theorem claims pointwise almost everywhere convergence, the construction indeed provides \(L^1\) convergence as well).

\begin{lemma}\label{steplem}
Any integrable function \([0,2\pi)^n \rightarrow \mathbb{R}\) can be approximated in \(L^1\) by a step function.
\end{lemma}

With this lemma, we can now, once again, prove Theorem \ref{torus1}.


\begin{proof}[Second proof of Theorem \ref{torus1}.]
Let \(\rho : [0, 2\pi)^n \rightarrow \mathbb{R}\) be a (representative of the \(L^1\)-) density of \((X_1, \ldots , X_n)\) with respect to Lebesgue measure. Then for \(m \in \mathbb{N}\), consider the map \(R_n^{(m)} : L^1([0, 2\pi)^n) \rightarrow L^1([0, 2\pi)^n)\) 
\[(R^{(m)}_n f)(\vec{x}) \:=\: \frac{1}{m^n} \sum_{k_1=0}^{m-1} \cdots \sum_{k_n=0}^{m-1} f\left( \frac{x_1 + 2\pi k_1}{m} , \ldots , \frac{x_n + 2\pi k_n}{m} \right) \, , \]
Then note that Lemma \ref{rlem} generalizes to \(n\)-dimensions, so \(R^{(m)}_n \rho\) is the density of \((m X_1, \ldots , m X_n)\), and \(R^{(m)}_n\) is a contraction.  Using Lemma \ref{steplem}, approximate \(\rho\) in \(L^1\) by a step function \(s : [0, 2\pi)^n \rightarrow \mathbb{R}\), so that \(\|\rho - s\|_1 < \epsilon\). Then for any \(\vec{x}\), \((R^{(m)}_n s) (\vec{x})\) is an approximation of the Riemann integral of \(s\). Consequently, we have that, for every \(\vec{x}\), \((R^{(m)}_n s)(\vec{x}) \rightarrow \|s\|_1\). Since \(s\) is a step function, it is bounded, so \(s \leq c\) for some \(c \in \mathbb{R}\). But this implies that \(R^{(m)}_n s \leq c\) as well, so by the bounded convergence theorem, \(R^{(m)}_n s \rightarrow \|s\|_1\) in \(L^1\). Hence, we may choose \(m\) sufficiently large such that \(\big \|R^{(m)}_n s - \|s\|_1 \big\|_1 < \epsilon\).

Then we have
\begin{align*}
\big\|R^{(m)}_n \rho - \|\rho\|_1 \big\|_1 & \:\leq\: \big\| R_n^{(m)} \rho - R^{(m)}_n s \big\|_1 + \big \|R^{(m)}_n s - \|s\|_1 \big \|_1 + \big\| \|s\|_1 - \|\rho\|_1 \big\|_1  \\
& \:\leq\: \|\rho - s\|_1 + \big\|R^{(m)}_n s - \|s\|_1 \big\|_1 + \|\rho - s\|_1 \\
& \:<\: 3 \epsilon \, .
\end{align*}

Thus, \(R^{(m)}_n \rho\), the joint-density of \((mX_1, \ldots, mX_n)\), converges in \(L^1\) to the constant function \(\|\rho\|_1 = 1\).

The \(L^1\) convergence of the joint-densities is sufficient to conclude the convergence of the measures. This is because, for any measurable \(A \subseteq [0,2\pi)^n\),
\[\nu^{(m)}(A) \:=\: \int_A R^{(m)}_n \rho \:\xrightarrow{m \rightarrow \infty}\: \int_A 1 \:=\: \mu_{[0,2\pi)^n}(A) \, .\]
\end{proof}

\subsection{Eigenvalue distributions}

We can reinterpret Theorem \ref{torus1} for elements on a torus. For the rest of this section, fix a Lie group isomorphism \(z = (z_1, \ldots , z_n) : T \rightarrow (\mathbb{S}^1)^n\). Then we may view each \(z_j\) as a projection onto a copy of \(\mathbb{S}^1\). Then, by considering this map, along with equating \(\mathbb{S}^1\) with \([0,2\pi)\) and  including the observation in Remark \ref{lem3p15}, we may restate Theorem \ref{torus1} as the following.

\begin{theorem}\label{torus2}
If \(Z\) is a random element of any torus \(T \cong (\mathbb{S}^1)^n\) with distribution absolutely continuous with respect to \(\mu_T\), then \(Z^m \Rightarrow Y\), a random variable with distribution \(\mu_T\). Moreover, if the density of \(Z\) with respect to \(\mu_T\) consists of a finite Laurent polynomial in the \(z_j\)'s, then \(Z^m\) will be exactly distributed as \(\mu_T\) for \(m\) greater than the highest degree of any \(z_j\).
\end{theorem}

\begin{remark}
This is enough information to already prove our eigenvalue result in the case of \(\mathbb{U}(N)\). If \(U\) has a distribution that is absolutely continuous with respect to \(\mu_{\mathbb{U}(N)}\), and if \((Z_1, \ldots , Z_N)\) are the eigenvalues of \(U\), then the distribution of \((Z_1, \ldots , Z_N)\) is absolutely continuous with respect to \(\mu_{(\mathbb{S}^1)^N}\), and eigenvalues of \(U^m\) are given as \((Z_1^m , \ldots , Z_N^m)\), which according to Theorem \ref{torus2} converges in distribution to a uniformly distributed random variable on \((\mathbb{S}^1)^N\). This is not sufficient for other groups. For example, the eigenvalues of elements in \(\mathbb{SO}(2n+1)\) are in complex-conjugate pairs, so eigenvalue distributions of random elements of \(\mathbb{SO}(2n+1)\) are not absolutely continuous with respect to \(\mu_{(\mathbb{S}^1)^{2n+1}}\). Instead, we need to investigate the ``degrees of freedom'' of the eigenvalues becoming uniformly distributed for large \(m\).
\end{remark}

For the following definition, recall that we let \(T\) be a maximal torus of \(L\), and the map \(\psi\) as defined in \eqref{psimap}.

\begin{definition}\label{def}
We say that the random pair \((U_{L/T}, U_T) \in L/T \times T\) is a \emph{random preimage of \(U\)} if \((U_{L/T}, U_T) \in \psi^{-1}(U)\) almost surely.
\end{definition}

\begin{example}\label{equirandpre}
For almost all \(x \in L\), \(\psi^{-1}(x)\) consists of \([N(T):T]\) elements. Provided that this is almost surely the case for \(\psi^{-1}(U)\) ,  by ignoring outcomes on a set of measure 0, we may define a random preimage \((U_{L/T}, U_T)\) by choosing uniformly amongst the \([N(T):T]\) preimages of \(U\) (and independently of \(U\)).
\end{example}

\begin{example}\label{detrandpre}
Suppose that there exists a region \(\mathcal{O} \subseteq T\) such that each set \(v \mathcal{O} v^{-1}\) for \(v \in N(T)\) is disjoint, and if \(U \in \psi(\bigcup_{v \in N(T)} v \mathcal{O} v^{-1})\) almost surely, then for almost all values of \(U\), we may choose \((U_{L/T}, U_T)\) such that \(U_T\) is almost surely in \(\mathcal{O}\). Thus, the preimage is almost surely determined by the value of \(U\).

For example, for \(L = \mathbb{U}(N)\), and \(T\) the subset of diagonal matrices, we may set
\[\mathcal{O} \:=\: \left\{ \left( \begin{array}{ccc} e^{i\theta_1} &  & 0 \\  & \ddots &  \\ 0 &  & e^{i\theta_N} \end{array} \right) \::\: 0 \leq \theta_1 < \ldots < \theta_N < 2\pi \right\} \:\subseteq\: T \, .\]
Then as long as the eigenvalues of \(U\) are almost surely distinct, we may define \((U_{L/T}, U_T)\) as the unique preimage of \(U\) for which \(U_T \in \mathcal{O}\).
\end{example}

For the remainder of this paper, let \((U_{L/T}, U_T)\) be any random preimage of \(U\). Note that \(\psi(U_{L/T}, U_T) = U\) almost surely, and \(U_T\) is a random element of \(T\) that is almost surely conjugate to \(U\).   As such, the distribution of the eigenvalues of \(U^m\) is equivalent to that of \((U_T)^m\). Hence, we can characterize the eigenvalues of high powers of \(U\) by characterizing high powers of \(U_T\).

As with Theorem \ref{torus2}, in order for the main results of this paper, Theorem \ref{mainevalthm} and Theorem \ref{lastthm}, to hold, we need a type of absolute continuity on the distribution of \(U\). Specifically, the criterion of interest is the absolute continuity of the distribution of \(U_T\) with respect to \(\mu_T\). Since \(U_T\) determines the eigenvalues of \(U\), and since \(T\) has the same degrees of freedom as the eigenvalues of elements in \(L\) (see Lemma \ref{toruspolys}), we may consider this a type of absolute continuity on the eigenvalue component of \(U\). However, this is \emph{not} the same as saying that the distribution of eigenvalues is absolutely continuous with respect to \(\mu_{(\mathbb{S}^1)^N}\).

There is a natural concern when dealing with this criterion, namely whether \(Law(U_T) \ll \mu_T\) depends on the choice of random preimage \((U_{L/T}, U_T)\). As it turns out, this property is independent of the choice of \((U_{L/T}, U_T)\). This result is convenient to show with extra tools, and thus is addressed in Theorem \ref{nodept} in section 4. At the moment, we can at least offer some assurance through Theorem \ref{torusabscty}.

\begin{theorem}\label{torusabscty}
If \(Law(U) \ll \mu_L\), then \(Law(U_T) \ll \mu_T\).
\end{theorem}

\begin{proof}
Suppose that \(A \subseteq T\) is such that \(\mu_T(A) = 0\). Then define \(B = \psi(L/T \times  A) = \{\ell a \ell^{-1} \;:\: \ell \in L, a \in A\}\). Then \((\mu_{L/T} \otimes \mu_T)(L/T \times A) = 0\). Since \(\psi\) is a smooth map between manifolds of the same dimension,  \(B\)  must have outer measure 0 (see \cite[Theorem 6.9]{lee}, noting that the concept of ``measure zero'' in that context corresponds to a set having outer measure 0 with respect to \(\mu_{L/T} \otimes \mu_T\) and \(\mu_L\)), so that \(B \subseteq B'\) for some Borel measurable \(B'\) with \(\mu_L(B') = 0\). Since \(U_T\) is conjugate to \(U\), we may say \(U_T \in A\) implies that \(U \in B\) (and \(U \subseteq B'\)), so that
\[P(U_T \in A) \:\leq\: P(U \in B') \, .\]
By the absolute continuity of the distribution of \(U\), the right-hand side of the above must be 0, so that the left-hand side is also 0, which proves the desired absolute continuity.
\end{proof}

We now state a simplified version of \cite[Lemma 2.6]{rains}, which highlights the connection between \(T\) and eigenvalues of elements in \(L\).

\begin{lemma}\label{toruspolys}
Let \(L\) be a compact Lie group with maximal torus \(T\) of dimension \(n\) and unitary representation \(\phi: L \rightarrow GL(V)\) with \(\text{dim}(V) = N\). Then there exists a set of monomials \(\{p_j\}_{1 \leq j \leq N}\) such that the eigenvalues of an element \(t \in T\) are given as
\[\Big\{p_j\big(z_1(t), \ldots, z_n(t)\big)\Big\}_{1 \leq j \leq N} \, .\]
\end{lemma}

\begin{proof}[Proof (sketch).]
Any irreducible unitary representation of \((\mathbb{S}^1)^n\) can be realized as a homomorphism \((\mathbb{S}^1)^n \rightarrow \mathbb{S}^1\), which is necessarily a Laurent monomial in each coordinate, that is, a map of the form \((x_1, \ldots, x_n) \mapsto x_1^{k_1} \ldots x_n^{k_n}\). Then for any unitary representation \(\phi : T \rightarrow GL(V)\), we may define a unitary representation \(p : (\mathbb{S}^1)^n \rightarrow GL(V)\) as \(p = \phi \circ z^{-1}\). Then \(p\) is a product of irreducible representations, which implies that \(p\) can be realized as a homomorphism \(p = (p_1, \ldots , p_N)  : (\mathbb{S}^1)^n \rightarrow (\mathbb{S}^1)^N\). Then each \(p_j\) is a Laurent monomial, and the eigenvalues of \(p(x_1, \ldots, x_n)\) are given as the set \(\{p_j(x_1, \ldots, x_n)\}_{1 \leq j \leq N}\). Thus, for any \(t \in T\), the eigenvalues of \(\phi(t) = p(z(t))\) are given as \(\{p_j(z_1(t), \ldots , z_n(t))\}_{1 \leq j \leq N}\).
\end{proof}

The monomials described in Lemma \ref{toruspolys} are precisely the monomials that come into Theorem \ref{rains1}. For example, the monomials for \(\mathbb{SU}(N)\) can be realized as \(p_j = z_j\) for \(1 \leq j \leq N-1\), and \(p_N = \overline{z_1 \ldots z_{N-1}}\).

We may now put all of our accumulated facts together into a proof of Theorem \ref{mainevalthm}.

\begin{proof}[Proof of Theorem \ref{mainevalthm}]

Define \(Z_j = z_j(U_T)\) and \(Z = (Z_1, \ldots , Z_n) = (z_1(U_T), \ldots , z_n(U_T))\). By the absolute continuity of the distribution of \(U_T\) with respect to \(\mu_T\), since \(z\) is a diffeomorphism, we may deduce that the distribution of \(Z\) is absolutely continuous with respect to \(\mu_{(S^1)^n}\). Then take Laurent monomials \(\{p_j\}_{1 \leq j \leq N}\) as in Lemma \ref{toruspolys} and \(p = (p_1, \ldots , p_N)\),

If \(Y \in (\mathbb{S}^1)^n\) is distributed as \(\mu_{(\mathbb{S}^1)^n}\), then by Theorem \ref{torus2}, \(Z^m \Rightarrow Y\), so  \(p(Z^m) \Rightarrow p(Y)\). Or, written more explicitly,
\[(p_1(Z_1^m, \ldots , Z_n^m), \ldots, p_N(Z_1^m, \ldots, Z_n^m)) \Rightarrow (p_1(Y_1, \ldots , Y_n) , \ldots , p_N(Y_1, \ldots , Y_n)) \, .\]
Therefore, the eigenvalue distribution of \((U_T)^m\), which is equivalent to the eigenvalue distribution of \(U^m\), converges weakly to the eigenvalue distribution of \(H_L^D\).

To prove the second part of the theorem,  we use the second part of Theorem \ref{torus2}, which tells us that if \(m\) is greater than the highest degree of any \(z_j\) in the density of \(U_T\), then \(Z^m\) is distributed as \(Y\), for which we will see \(p(Z^m)\) is distributed as \(p(Y)\), so the eigenvalue distribution of \(U^m\) will be distributed as that of \(H_L^D\).
\end{proof}

\section{Limiting distributions of powers of group elements}

This section considers the distribution of powers of a random Lie group element itself, which results in statements stronger than those regarding distributions of eigenvalues. There is one result of this nature that we may already deduce: if the distribution of \(U\) conjugate-invariant and \(Law(U) \ll \mu_L\), then \(U^m\) converges in distribution to \(H_L^D\). To offer a brief justification, if \(U\) has conjugate-invariant distribution, then the distribution of \(U\) is determined by the distribution of the eigenvalues, as can be said of \(U^m\), as well as its limit in \(m\). Hence, if the eigenvalues achieve convergence in distribution, so must \(U^m\), and it must converge to \(H_L^D\). A rigorous justification of this fact will follow from the more general result of this section, Theorem \ref{lastthm}.

Recall the definition of random preimage in Definition \ref{def}. While the previous section and the discussion above focused on \(U_T\), our results of this section will require the distribution of \(U_{L/T}\) to be considered.  Lemma \ref{lastlem}, stated on its own for the sake of discussion, hints at the connection between \(U_{L/T}\) and the limiting distribution of \(U^m\).

\begin{lemma}\label{lastlem}
\((U_{L/T}, U_T^m)\) is a random preimage of \(U^m\).
\end{lemma}

\begin{proof}
Note that, for any \(a \in L/T\) and \(b \in T\), \(\psi(a, b)^m = \psi(a, b^m)\). Then \(\psi(U_{L/T}, U_T^m) = \psi(U_{L/T}, U_T)^m = U^m\) almost surely.
\end{proof}

\begin{remark}\label{thmnotobv2}
Based on this lemma and Theorem \ref{torus2}, it may be easy to think that \(U^m \sim \psi(U_{L/T}, U_T^m)\) converges in distribution to \(\psi(U_{L/T}, Y)\), where \(Y\) is uniform on \(T\) and independent of \(U_{L/T}\). However, Lemma \ref{lastlem} on its own does not imply this, neither guaranteeing independence in the limit,  nor proving that such a limit exists in the first place. Nevertheless, independence will be achieved in the limit, as stated in the last main result, Theorem \ref{lastthm}, but as with the independence in Theorem \ref{torus1}, it requires more careful reasoning to show. For this result, we will once again use the condition \(Law(U_T) \ll \mu_T\), and in particular we do not require \(Law(U) \ll \mu_L\), which will be highlighted in the example to follow. Possibly more surprising is that the statement of Theorem \ref{lastthm} implies that the distribution of \(\psi(U_{L/T}, Y)\) is the same regardless of choice of random preimage \((U_{L/T}, U_T)\). Indeed this can be proven true independent of the methods used here, and is addressed in detail in section 4, Theorem \ref{nodepl}.
\end{remark}

\begin{proof}[Proof of Theorem \ref{lastthm}]

We will show this by first showing that \((U_{L/T},U_T) \Rightarrow (U_{L/T},Y)\). Fix a measurable set \(A \subseteq L/T\). Suppose that \(P(U_{L/T} \in A) > 0\). Then consider  \(\frac{P((U_{L/T},U_T) \in A \times \cdot)}{P(U_{L/T} \in A)}\) as a measure on \(T\). In fact, this is the probability distribution of \(U_T\) conditioned on \(U_{L/T} \in A\), which we will define as \(V_A\). Or more precisely, define \(\Omega' = \{U_{L/T} \in A\}\) and \(P' = \frac{P((U_{L/T} \in A) \cap \cdot)}{P(U_{L/T} \in A)}\), so that \((\Omega', P')\) is a probability space, and then define \(V_A := U_T|_{\Omega'}\). It should be noted that \(V_A\) must have distribution absolutely continuous with respect to \(\mu_T\), because if \(B \subseteq T\) where \(\mu_T(B) = 0\), then 
\[P(V_A \in B) \cdot P(U_{L/T} \in A) \:=\: P((U_{L/T},U_T) \in A \times B) \:\leq\: P(U_T \in B) \:=\: 0 \, .\]
where the last equality follows from the absolute continuity of the distribution of \(U_T\) with respect to \(\mu_T\). Thus, \(P(V_A \in B) = 0\), proving the absolute continuity of the distribution of \(V_A\) with respect to \(\mu_T\). Next, note that the distribution of \((V_A)^m\) is given as \(\frac{P((U_{L/T}, (U_T)^m) \in A \times \cdot)}{P(U_{L/T} \in A)}\). If we define \(Y\) to be uniform on \(T\) and independent of \(U_{L/T}\), then using Theorem \ref{torus2}, \((V_A)^m\) converges in distribution to \(Y\), so for any measurable \(B \subseteq T\), \(\frac{P((U_{L/T}, (U_T)^m) \in A \times B)}{P(U_{L/T} \in A)} \rightarrow P(Y \in B)\), or in other words,
\begin{equation}\label{probconv}
P((U_{L/T}, (U_T)^m) \in A \times B) \rightarrow P(U_{L/T} \in A) P(Y \in B) = P((U_{L/T}, Y) \in A \times B) \, .
\end{equation}

If \(P(U_{L/T} \in A) = 0\), then \(P((U_{L/T}, (U_T)^m) \in A \times B) \leq P(U_{L/T} \in A)\) implies that the convergence in \eqref{probconv} trivially occurs, so it must occur for all measurable \(A\). Since \(\{A \times B : A \subseteq L/T, B \subseteq T \text{ measurable}\}\) is a \(\pi\)-system that generates the Borel \(\sigma\)-algebra of \(L/T \times T\), the convergence in \eqref{probconv} occurs for all measurable subsets of \(L/T \times T\). Hence,  \((U_{L/T}, (U_T)^m)\) converges to \((U_{L/T}, Y)\) in distribution. By Lemma \ref{lastlem}, \(U^m \sim \psi(U_{L/T}, (U_T)^m)\), which must converge in distribution to \(\psi(U_{L/T} , Y)\).

\end{proof}

We now prove the aforementioned corollary.

\begin{corollary}\label{cor1}
If the distribution of \(U\) is conjugate-invariant and absolutely continuous with respect to \(\mu_L\), then \(U^m\) converges in distribution to \(H_L^D\).
\end{corollary}

\begin{proof}
If \(U\) is conjugate-invariant, then for any \(g \in L\), \(gUg^{-1} \sim U\). Let \((U_{L/T}, U_T)\) be a uniform random preimage of \(U\) as described in Example \ref{equirandpre}. Then, almost surely, \(\psi(U_{L/T}, U_T) = U\) and \(\psi(g U_{L/T}, U_T) = gUg^{-1}\), so it can be seen that \((g U_{L/T}, U_T)\) is a uniform random preimage of \(g U g^{-1}\). The distribution of a uniform random preimage is determined by the distribution of the image, so \(g U g^{-1} \sim U\) implies that \((g U_{L/T}, U_T) \sim (U_{L/T}, U_T)\), and in particular \(g U_{L/T} \sim U_{L/T}\), which forces \(U_{L/T}\) to have \(\mu_{L/T}\) as its distribution by the uniqueness of Haar measure. By Theorem \ref{lastthm}, \(U^m\) must  converge in distribution to  \(\psi(U_{L/T}, Y)\), where \(Y\) is independent of \(U_{L/T} \sim \mu_{L/T}\), which is indeed the distribution of \(H_L^D\). 
\end{proof}

We end with an example, considering a random variable that does not have absolutely continuous distribution with respect to \(\mu_L\). In particular, this construction will be such that \(U_{L/T}\) consists of point masses. In the example, we will explicitly compute \(U_T\) and \(U_{L/T}\), using a uniform random preimage, as described in Example \ref{equirandpre}, and use it to compute the limiting distribution of \(U^m\).

\begin{example}\label{lastexample}
Consider \(L = \mathbb{U}(2)\), and \(T\) the set of diagonal matrices of \(\mathbb{U}(2)\). Let \(D_1\) and \(D_2\) be random elements of \(T\) with distributions absolutely continuous with respect to \(\mu_T\). Let \(X\) be a random variable, independent of \(D_1\) and \(D_2\), with \(P(X = 0) = P(X = 1) = \frac{1}{2}\). Define
\[a = \left( \begin{array}{cc} \sqrt{2}/2  &  -\sqrt{2}/2  \\ \sqrt{2}/2 &  \sqrt{2}/2 \end{array} \right)  \quad,\quad p = \left( \begin{array}{cc} 0  &  1  \\ 1 &  0 \end{array} \right) \, .\]
Note that \(ap\) swaps the 2 columns of \(a\), while  for \(t \in T\), 
\(ptp^{*}=ptp\) swaps the diagonal entries of \(t\).

Define \(U = X D_1 + (1-X) a D_2 a^*\). We will now compute the distribution of a uniform random preimage \((U_{L/T}, U_T)\).  Note that \(P(U \in T \cup aTa^*) = 1\). If \(t \in T\) and is a regular value of \(\psi\) (or in other words, if \(t\) has distinct diagonal entries), then letting \(e \in L\) denote the identity, then \(\psi^{-1}(t) = \{(eT, t), (pT, ptp)\}\) and \(\psi^{-1}(ata^*) = \{(aT, t) , (apT, ptp)\}\). If we let \(\widehat{X}\) be an independent random variable with \(P(\widehat{X}=0) = P(\widehat{X}=1) = \frac{1}{2}\), then \(U_T\) is distributed as
\[X \widehat{X} D_1 +  X (1-\widehat{X}) p D_1 p  + (1-X) \widehat{X} D_2 + (1-X)(1-\widehat{X}) p D_2 p \, .\]
Note that this has absolutely continuous distribution with respect to \(\mu_T\), which implies that Theorem \ref{lastthm} holds, so \(U^m\) converges in distribution to \(\psi(U_{L/T}, Y)\) where \(Y\) is independent and distributed as \(\mu_T\). To compute \(\psi(U_{L/T}, Y)\), we see that \(U_{L/T}\) satisfies 
\[P(U_{L/T} = eT) \:=\: P(U_{L/T} = pT) \:=\: P(U_{L/T} = aT) \:=\: P(U_{L/T} = apT) \:=\: \frac{1}{4} \, .\]
Then by noting that \(pYp \sim Y\), and thus \(\widehat{X} Y + (1-\widehat{X}) pYp \sim Y\), we have
\begin{align*}
\psi(U_{L/T}, Y) &\:\sim\: X \widehat{X} Y +  X (1-\widehat{X}) p Y p  + (1-X) \widehat{X} aYa^* + (1-X)(1-\widehat{X}) ap Y pa^* \\
& \:\sim\: XY + (1-X)aYa^* \, .
\end{align*}
\end{example}

Thus, we may explicitly describe limits of distributions of random elements of compact Lie groups with weaker conditions than absolute continuity with respect to \(\mu_L\).

\section{Regarding random preimages}

In stating our theorems, we chose to discuss arbitrary random preimages, from opposed to merely making a fixed choice, in order to offer freedom and generality in applying these methods; there are reasons that the preimages described in Examples \ref{equirandpre} and \ref{detrandpre} are both useful. However, the provided generality may cause further questions.

There are two facts that this section will address, demonstrating that the assumptions and consequences of our main results, Theorem \ref{mainevalthm} and Theorem \ref{lastthm}, do not depend on the choice of random preimage \((U_{L/T}, U_T)\). We will first show Theorem \ref{nodept}, which states that the property \(Law(U_T) \ll \mu_T\) does not depend on the choice of random preimage. Secondly, we will show Theorem \ref{nodepl}, which states that if \(Law(U_T) \ll \mu_T\), and if \(Y\) is distributed as \(\mu_T\) and independent of \(U_{L/T}\), then the distribution of \(\psi(U_{L/T}, Y)\) does not depend on the choice of random preimage. It should be mentioned that Theorem \ref{nodepl} is technically a consequence of Theorem \ref{lastthm}, even though Theorem \ref{lastthm} offers no intuition whatsoever as to why this is true. Thus, this section will offer a proof of this independent of the methods used in Theorem \ref{lastthm}.

In order to show these facts, it will be convenient to use an extra tool, namely the Weyl group, which will now be described based on \cite{applebaum} and \cite{hall}. We denote the normalizer of \(T\) in \(L\) as
\[N(T) = \{a \in L \::\: a t a^{-1} \in t \text{ for all } t \in T\} \, .\]
We call the group \(N(T)/T\) the \emph{Weyl group} of \(L\). Note that \([N(T):T]\) is always finite. The Weyl group acts on \(T\) by conjugation, that is, if \(vT \in N(T)/T\) and \(t \in T\), then \(vT \cdot t := v t v^{-1}\). We will also need the fact that the map \(T \rightarrow T\) defined as \(t \mapsto vT \cdot t \) is not only a diffeomorphism, but also preserves the measure \(\mu_T\) by the translation invariance of \(\mu_T\).

The Weyl group also acts on \(L/T\) via right multiplication, so for \(vT \in N(T)/T\) and \(aT \in L/T\), \(vT \cdot aT := aT v^{-1} = av^{-1} T\). Then we may define an action of \(N(T)/T\) on \(L/T \times T\), namely
\begin{equation}\label{weylaction}
vT \cdot (aT, t) \: :=\: (a v^{-1} T \:,\: v t v^{-1}) \, .
\end{equation}

This action will preserve the value of \(\psi\), so that if \(w \in N(T)/T\), then \(\psi(w \cdot (aT, t)) = \psi(aT, t)\). Phrased another way, if \((aT, t) \in \psi^{-1}(x)\), then \(w \cdot (aT, t) \in \psi^{-1}(x)\).  If we know \(x = \psi(aT, t)\), where \(t\) is an element of \(T'\) as described in Theorem \ref{teeprime}, then the Weyl group will act freely and transitively on \(\psi^{-1}(x)\). This result can be mostly reconstructed from results in \cite{hall}, specifically Proposition 11.4 and Lemma 11.26.

\begin{lemma}\label{teeprime}
There exists a Borel measurable set \(T' \subseteq T\) satisfying the following

\begin{enumerate}[(i)]
\item \(\mu_T(T') = 1\).
\item \(T'\) is invariant under the action of \(N(T)/T\) defined in \eqref{weylaction}
\item For all \(t \in T'\) and all nontrivial \(w \in N(T)/T\), \(w \cdot t \neq t\).
\item If \(t \in T'\), \(s \in T\) and \(\psi(aT, t) = \psi(bT, s)\), then there exists a unique \(w \in N(T)/T\) such that \(w \cdot (aT,t) = (bT,s)\), and \(s \in T'\).
\end{enumerate}
\end{lemma}

\begin{proof}
If \(\overline{\langle t \rangle}\) is the closure of the subgroup generated by \(t\), then set \(T' = \{t \in T \::\: \overline{\langle t \rangle} = T\}\). Then \(T \setminus T'\) consists of countably-many subtori of dimension less than the dimension of \(T\). To see this, first take a Lie group isomorphism \(z = (z_1, \ldots , z_n) : T \rightarrow \mathbb{S}^n\), then observe that \(z(T')\) is the set of tuples \((e^{2\pi i \theta_1} , \ldots , e^{2\pi i \theta_n})\) such that \(1 , \theta_1 , \ldots , \theta_n\) are \(\mathbb{Q}\)-linearly independent (see \cite[Proposition 11.4]{hall}). Then  \(z(T) \setminus z(T')\) corresponds to a set of  tuples \((\theta_1, \ldots , \theta_n) \in [0,1)^n\) such that there exist \(q_0, \ldots , q_n \in \mathbb{Q}\) where \(q_0 + q_1 \theta_1 + \ldots + q_n \theta_n = 0\), meaning that this set of tuples in \([0,1)^n\) is a countable union of closed hyperplanes of smaller dimension. We may then deduce that \(z(T) \setminus z(T')\) is a union of countably-many subtori of smaller dimension, so the same can be said of \(T \setminus T'\). Therefore, \(T'\) is a measurable set and of full measure.

If \(t \in T'\), and if \(w = vT \in N(T)/T\), then
\[\overline{\langle w \cdot t \rangle} \:=\: \overline{\langle v t v^{-1} \rangle } \:=\: v \overline{\langle t \rangle } v^{-1} \:=\: v T v^{-1} \:=\: T \, ,\]
so \(w \cdot t \in T'\) as well. And if \(w \cdot t = t\), then \(v t v^{-1} = t\), and for any \(n \in \mathbb{Z}\), \(v t^n v^{-1} = t^n\). Then the diffeomorphism \(t \mapsto w \cdot t\) fixes \(\langle t \rangle\), a dense subset of \(T\), so it must fix all of \(T\). Equivalently, we can say that \(v\) commutes with every element of \(T\). By the maximality of \(T\), we must have \(v \in T\) (or else \(v\) and \(T\) would be contained in a torus larger than \(T\)), so that \(w = vT = eT\).

Lastly, if \(t \in T'\) and \(s \in T\) where \(\psi(aT, t) = \psi(bT, s)\), then \(ata^{-1} = bsb^{-1}\), and \(t = (a^{-1} b) s (a^{-1} b)^{-1}\), so that
\[T \:=\: \overline{\langle t \rangle } \:=\: (a^{-1} b) \overline{\langle s \rangle } (a^{-1} b)^{-1} \:\subseteq\: (a^{-1}b) T (a^{-1}b)^{-1} \, .\]
Since \((a^{-1}b) T (a^{-1}b)^{-1}\) is a torus, by the maximality of \(T\), we must have \(T = (a^{-1} b) T (a^{-1} b)^{-1}\), so that \(a^{-1} b \in N(T)\), implying that \((b^{-1} a) T \in N(T)/T\), and
\[(b^{-1} a) T \,\cdot\, (aT, t) \:=\: (a (a^{-1} b) T,  (b^{-1} a) t (a^{-1} b)) \:=\: (bT, s) \,.\]
It follows that \(s \in T'\).
\end{proof}

Lemma \ref{teeprime} and the commentary above tell us that the Weyl group action can often map one preimage of a fixed value \(x\) to another. The next result tells us that a \emph{random eleme}nt of the Weyl group can be used to convert between r\emph{andom preim}ages.

\begin{lemma}\label{weylconv}
If \((U_{L/T}^{(1)}, U_{T}^{(1)})\) and \((U_{L/T}^{(2)}, U_T^{(2)})\) are both random preimages of \(U\),  and if \(Law(U_T^{(1)}) \ll \mu_T\), then there exists a random element of the Weyl group \(W \in N(T)/T\) such that, almost surely,
\[W \cdot (U_{L/T}^{(1)},  U_T^{(1)}) = (U_{L/T}^{(2)}, U_T^{(2)}) \, .\]
\end{lemma}

\begin{proof}
Take \(T' \subseteq T\) as in Lemma \ref{teeprime}. Then for any \(w \in N(T)/T\), define 
\[\mathfrak{P}_w = \{((u, t), w \cdot (u,t)) \::\:  u \in L/T \:,\: t \in T'\} \:\subseteq\: (L/T \times T')^2 \,.\] Then each \(\mathfrak{P}_w\) is Borel measurable. Furthermore, if \(v, w \in N(T)/T\) are such that \(((aT,t), (bT,s)) \in \mathfrak{P}_w \cap \mathfrak{P}_v\), then \(w \cdot t = s = v \cdot t\), and \((v^{-1} w) \cdot t = t\), and since \(t \in T'\), we must have \(w = v\). Hence, the collection \(\{\mathfrak{P}_w\}_{w \in N(T)/T}\) is disjoint. Then set \(\mathfrak{P} = \bigcup_{w \in N(T)/T} \mathfrak{P}_w\), and define \(\omega : \mathfrak{P} \rightarrow N(T)/T\) as \(\omega((aT,t), w \cdot (aT, t)) = w\). Then for any \(w \in N(T)/T\), \(\omega|_{\mathfrak{P}_w} \equiv w\), so we see that \(\omega\) is well-defined and measurable. Then observe that, for any \((aT, t) \in L/T \times T'\) and \(w \in N(T)/T\),
\[\omega((aT,t), w\cdot (aT,t)) \cdot (aT,t) \:=\: w \cdot (aT,t) \, .\]

By the absolute continuity of \(U_T^{(1)}\), we have that \(P(U_T^{(1)} \in T') = 1\). Then almost surely \(\psi(U_{L/T}^{(1)} , U_T^{(1)}) = U = \psi(U_{L/T}^{(2)} , U_T^{(2)})\), so by Lemma \ref{teeprime}, \(P(U_T^{(2)} \in T')=1\), and \(\big( (U_{L/T}^{(1)} , U_T^{(1)}) , (U_{L/T}^{(2)} , U_T^{(2)}) \big) \in \mathfrak{P}\) almost surely, so that we may give the almost sure definition of \(W \in N(T)/T\) as \(W = \omega\big( (U_{L/T}^{(1)} , U_T^{(1)}) , (U_{L/T}^{(2)} , U_T^{(2)}) \big)\). Thus, we see that \(W \cdot (U_{L/T}^{(1)} , U_T^{(1)}) = (U_{L/T}^{(2)} , U_T^{(2)})\) almost surely.

\end{proof}

We now use Lemma \ref{weylconv} to immediately prove Theorems \ref{nodept} and \ref{nodepl} in succession.

\begin{theorem}\label{nodept}
If \((U_{L/T}^{(1)}, U_{T}^{(1)})\) and \((U_{L/T}^{(2)}, U_T^{(2)})\) are both random preimages of \(U\), then the distribution of \(U_T^{(1)}\) is absolutely continuous with respect to \(\mu_T\) if and only if the distribution of \(U_T^{(2)}\) is.
\end{theorem}

\begin{proof}
Suppose that \(U_T^{(1)}\) has distribution absolutely continuous with respect to \(\mu_T\). Then by Lemma \ref{weylconv}, there exists a random \(W \in N(T)/T\) such that \(W \cdot U_T^{(1)} = U_T^{(2)}\). Then if \(\mu_T(A) = 0\), then
\begin{align*}
P(U_T^{(2)} \in A) &\:=\: P(W \cdot U_T^{(1)} \in A) \\
& \:=\: \sum_{w \in N(T)/T} P(w \cdot U_T^{(1)} \in A \:,\: W = w) \\
& \:=\: \sum_{w \in N(T)/T} P(U_T^{(1)} \in w^{-1} \cdot A \:,\: W = w) \\
& \:\leq\: \sum_{w \in N(T)/T} P(U_T^{(1)} \in w^{-1} \cdot A) \:=\: 0 \, .
\end{align*}
This proves the absolute continuity of the distribution of \(U_T^{(2)}\) with respect to \(\mu_T\).
\end{proof}

\begin{theorem}\label{nodepl}
Suppose that \((U_{L/T}^{(1)}, U_{T}^{(1)})\) and \((U_{L/T}^{(2)}, U_T^{(2)})\) are both random preimages of \(U\), and \(U_T^{(1)}\) has absolutely continuous distribution with respect to \(\mu_T\), and \(Y\in T\) is distributed as \(\mu_T\) and is independent of \(U_{L/T}^{(1)}\) and \(U_{L/T}^{(2)}\). Then \(\psi(U_{L/T}^{(1)}, Y) \sim \psi(U_{L/T}^{(2)},Y)\). \\
\end{theorem}

\begin{proof}
By Lemma \ref{weylconv}, there exists a random \(W \in N(T)/T\) such that \(W \cdot U_{L/T}^{(1)} = U_{L/T}^{(2)}\). For \(w \in N(T)/T\),  \(w \cdot Y \sim Y\), because \(P(w \cdot Y \in A) = P(Y \in w^{-1} \cdot A ) = P( Y \in A)\).  Then we have, for any measurable \(A \subseteq L\),
\begin{align*}
P \Big ((\psi(U_{L/T}^{(2)}, Y) \in A \Big) & \:=\: P \Big(\psi(W \cdot U_{L/T}^{(1)}  \:,\: Y ) \in A \Big) \\
& \:=\: \sum_{w \in N(T)/T} P\Big(\psi(w \cdot U_{L/T}^{(1)} \:,\: Y) \in A \:,\: W = w \Big) \\
& \:=\: \sum_{w \in N(T)/T} P\Big(\psi(U_{L/T}^{(1)} \:,\: w^{-1} \cdot Y) \in A \:,\: W = w \Big) \\
& \:=\: \sum_{w \in N(T)/T} P\Big(\psi(U_{L/T}^{(1)}  \:,\: Y ) \in A \:,\: W = w \Big) \\
& \:=\: P\Big(\psi(U_{L/T}^{(1)}  \:,\: Y) \in A \Big) \, .
\end{align*}
This proves the claim.
\end{proof}

In conclusion, we may say with certainty that the choice of random preimage does not impact the conditions or results of Theorems \ref{mainevalthm} and \ref{lastthm}.







\providecommand{\bysame}{\leavevmode\hbox to3em{\hrulefill}\thinspace}
\providecommand{\MR}{\relax\ifhmode\unskip\space\fi MR }
\providecommand{\MRhref}[2]{%
  \href{http://www.ams.org/mathscinet-getitem?mr=#1}{#2}
}
\providecommand{\href}[2]{#2}

\end{document}